\documentclass[11pt,A4paper]{amsart}
\usepackage{color}

\definecolor{linkred}{RGB}{199,21,133}
\definecolor{linkblue}{RGB}{16, 78, 139}
\usepackage[left=32mm,right=32mm,top=30mm,bottom=32mm]{geometry}
\usepackage{amssymb,euscript,tikz,units}
\usepackage[colorlinks,citecolor=linkblue,linkcolor=linkred]{hyperref}
\usepackage{verbatim}
\usepackage[nameinlink]{cleveref}
\usepackage{colonequals}
\usepackage{tikz}
\usepackage{enumitem}

\newcommand{\calT}{\mathcal{T}}
\newcommand{\T}{\calT}
\newcommand{\M}{\mathcal{M}}
\newcommand{\sM}{\mathcal{M}}
\newcommand{\sA}{\mathcal{A}}

\newcommand{\R}{\mathbb{R}}

\renewcommand{\H}{\mathbb{H}}

\newcommand{\eps}{\varepsilon}
\newcommand{\wep}{Weil-Petersson}

\newcommand{\sbs}{\subset}
\newcommand{\limg}{\lim_{g\rightarrow\infty}}
\newcommand{\sT}{\mathcal{T}}
\newcommand{\dvol}{d\mathit{vol}}

\def\sys{\mathop{\rm sys}}
\def\area{\mathop{\rm Area}}

\def\Vol{\mathop{\rm Vol}}

\def\Prob{\mathop{\rm Prob}\nolimits_{\rm WP}^g}
\def\E{\mathop{\mathbb{E}_{\rm WP}^g}}

\def\diam{\mathop{\rm Diam}}

\DeclareMathOperator{\WP}{WP}

\newcommand{\ls}{\ell_{\sys}}
\newcommand{\lss}{\ell_{\sys}^{\rm sep}}
\newcommand{\ssys}{\ell_{\sys}^{\rm sep}}

\newcommand{\ewp}{\mathbb{E}_{\rm WP}^{g}}

\theoremstyle{plain}
\newtheorem{theorem}{Theorem}

\newtheorem{proposition}[theorem]{Proposition}
\newtheorem{lemma}[theorem]{Lemma}

\newtheorem{remark}[theorem]{Remark}

\newcommand{\be}{\begin{equation}}
\newcommand{\ene}{\end{equation}}
\newcommand{\br}{\begin{remark}}
	\newcommand{\er}{\end{remark}}
\newcommand{\bl}{\begin{lem}}
	\newcommand{\el}{\end{lem}}
\newcommand{\bcor}{\begin{cor}}
	\newcommand{\ecor}{\end{cor}}
\newcommand{\bpro}{\begin{pro}}
	\newcommand{\epro}{\end{pro}}
\newcommand{\ben}{\begin{enumerate}}
	\newcommand{\een}{\end{enumerate}}
\newcommand{\bp}{\begin{proof}}
	\newcommand{\ep}{\end{proof}}
\newcommand{\bpo}{\begin{pro}}
	\newcommand{\epo}{\end{pro}}
\newcommand{\beq}{\begin{equation*}}
\newcommand{\eeq}{\end{equation*}}
\newcommand{\bear}{\begin{eqnarray}}
\newcommand{\eear}{\end{eqnarray}}
\newcommand{\beqar}{\begin{eqnarray*}}
	\newcommand{\eeqar}{\end{eqnarray*}}
\newcommand{\bt}{\begin{theorem}}
	\newcommand{\et}{\end{theorem}}
\newcommand{\bex}{\begin{excer}}
	\newcommand{\eex}{\end{excer}}

\theoremstyle{definition}

\theoremstyle{remark}

\newtheorem*{rem*}{Remark}
\newtheorem*{def*}{Definition}
\newtheorem*{definition*}{Definition}
\newtheorem*{thm*}{\bf Theorem}

\begin{document}
	
\title{The simple separating systole for hyperbolic surfaces of large genus}

\author{Hugo Parlier, Yunhui Wu, and Yuhao Xue}
\thanks{H.P. supported by the Luxembourg National Research Fund OPEN grant O19/13865598.}
\thanks{Y.W. supported by a grant from Tsinghua University.}
\address{Department of Mathematics, University of Luxembourg, Esch-sur-Alzette, Luxembourg}
\email[(H.~P.)]{hugo.parlier@uni.lu}

\address{Yau Mathematical Sciences Center \& Department of Mathematical Sciences, Tsinghua University, Beijing, China}
\email[(Y.~W.)]{yunhui\_wu@tsinghua.edu.cn}
\email[(Y.~X.)]{xueyh18@mails.tsinghua.edu.cn}

\date{}
\maketitle

\begin{abstract}
In this note we show that the expected value of the separating systole of a random surface of genus $g$ with respect to Weil-Petersson volume behaves like $2\log g $ as the genus goes to infinity. This is in strong contrast to the behavior of the expected value of the systole which, by results of Mirzakhani and Petri, is independent of genus.
\end{abstract}

\section{Introduction}

Over the last couple of decades, different models for random hyperbolic surfaces which sample large genus behavior have been studied. Two models are particularly natural. The first, more combinatorial in nature, consists in randomly gluing triangles and looking at the resulting conformal class of hyperbolic metric \cite{BrooksMakover}. The second in using the Weil-Petersson volume form on the moduli space $\M_g$ of hyperbolic structures on a genus $g$ surface, and then sampling by choosing a random point in each genus and letting the genus grow. 

This latter model was studied by Mirzakhani \cite{Mirz13} who, using her pioneering results on Weil-Petersson volumes, showed a number of striking results, for example that random surfaces have small diameter but large embedded balls. She also studied the systole (the length of the shortest closed geodesic) and separating systole (the length of the shortest separating simple closed geodesic) showing that the former was ``small but not too small" and the latter grew at least like $\log(g)$. The behavior of lengths of curves was further explored by Mirzakhani and Petri in \cite{MirzakhaniPetri} where they show the number of short curves is asymptotically Poisson distributed, and as an application computed the expected value of the systole (which is an explicit positive constant). Note that this behavior is independent of genus. These results mirror similar results for the combinatorial model \cite{Petri, PetriThale}, reinforcing bridges between the two models \cite{GPY}, but also some key differences in the behavior of (very) short curves.

We focus on the Weil-Petersson model, and study the length of the separating systole $\lss(X)$ of a surface $X$:
$$
\lss(X)=\min\big\{\ell_\gamma(X)\,\mid\, \text{$\gamma\subset X$ is a separating simple closed geodesic}\big\}.
$$
Unlike the systole function, the separating systole function is unbounded on $\M_g$. This is not too difficult to see by considering a pants decomposition of a genus $g$ surface whose geodesics are all non-separating. By making these arbitrarily short, via the collar lemma, the length of any curve that crosses them is at least twice the collar width, and the observation follows. A related, but different, function on moduli space is given by the length of the shortest geodesic trivial in homology (but homotopically non-trivial), studied by Sabourau \cite{Sabourau} who showed a universal upper bound on the same order as the corresponding bound on systole (on the order of $\log(g)$). 

We study the expectation of this function. Let $\E[\lss]$ be the Weil-Petersson expected value of $\lss(\cdot)$ over $\M_g$:
\[\E[\lss]=\frac{\int_{\M_g}\lss(X)dX}{V_g}\]
where $V_g$ is the Weil-Petersson volume of $\M_g$. As mentioned previously, Mirzakhani's results showed that the expected values grow at least like $\log(g)$ as a function of genus. 

Building on Mirzakhani's methods, and essential probability estimates in \cite{NWX20}, we add the following asymptotic growth result to this panorama:

\begin{theorem}\label{mt-1}
Let $\lss(\cdot)$ be the separating systole function on $\M_g$. Then
\[\lim \limits_{g\to \infty}\frac{\E[\lss]}{2\log g}=1.\]
\end{theorem}

This result gives an affirmative answer to \cite[Question 49]{NWX20}. 

The fact that the expected value is at least on the order of $2\log(g)$ follows from Mirzakhani's estimates. The main point of this note is the upper limit, and uses the main result from \cite{NWX20} in an essential way. Indeed, in  \cite{NWX20} it is shown that that the separating systole is of length $< 2 \log (g)$ with probability tending to $1$ as the genus goes to infinity. This does not imply the expected value result however, because the separating systole is unbounded over $\M_g$. Our contribution in this paper is exactly to overcome this problem.

\subsection*{Plan of the paper.} In Section \ref{pre} we review the relevant background and provide an upper bound for $\lss(X)$ in terms of the systole and diameter of $X$. In Section \ref{sec-1/h} we prove two bounds for the integral of the reciprocal of the Cheeger constant over small subsets in moduli space. We complete the proof of Theorem \ref{mt-1} in Section \ref{Sec-mt-1}. 

\subsection*{Acknowledgements.}
The authors are grateful to Curtis McMullen and Alex Wright for the correspondence on the proof of  \cite[Theorem 4.4]{Mirz13}. 


\section{Preliminaries} \label{pre}

In this section, we set up notation and review relevant background material about hyperbolic surfaces and the Weil-Petersson metric on the moduli space of Riemann surfaces.

\subsection{Riemann surfaces}
Let $\Sigma_g$ be a closed orientable surface of genus $g\geq 2$. Let $\M_g$ be the moduli space of all complete hyperbolic metrics homeomorphic to $\Sigma_g$, up to isometry. Any $X\in \M_g$, by Gauss-Bonnet, has its area $\area(X)=4\pi(g-1)$. We are interested in simple closed geodesics of $X$ and sometimes in collections of disjoint simple closed geodesics which we call geodesic multicurves. The length of a geodesic multicurve $\gamma$ will be denoted $\ell_{\gamma}(X)$.

The \emph{systole} $\ls(X)$, that is the length of the (or a) shortest closed geodesic on $X$, satisfies an upper bound on the order $\ls(X)=O(\log g)$ where the implied constant is independent of $g$. For a closed surface, the systole is always a simple closed geodesic. Similarly, the shortest non-separating closed geodesic is always simple, but this may not be the case of the shortest {\it separating geodesic} (see the example hinted at in the introduction). In this paper, we study the \emph{separating systole} $\lss(X)$ of $X$, the length of the shortest separating and simple closed geodesic of $X$. 

The following estimate will be used in the sequel. We state it for any closed orientable Riemannian surface just to highlight that it has nothing to do with hyperbolic geometry (and the proof is the same). 
\begin{lemma}\label{lss<diam}
Let $X$ be a closed Riemannian surface of genus $g\geq 2$. Then the length $\ssys(X)$ of its separating systole satisfies
$$
\ssys(X)< 2\, \ls(X)+4\, \diam(X)
$$
where $\ls(X)$ is the systole length of $X$ and $\diam(X)$ is the diameter of $X$.
\end{lemma}
\begin{proof}
Recall that the systole of a closed genus $g$ surface is always a simple closed geodesic. If the systole is separating, then the inequality clearly holds, so we suppose that the systole is non-separating, hence non-trivial in homology.

Now consider a minimal length homology basis (for the sum of lengths) of $X$ that contains a systole $\gamma_1$. In \cite[Section 5]{Gromov}, it is shown that such a basis $\mathcal{B}$ always exists and all of its curves are geodesically {\it convex} subsets. By this we mean that for any curve $\gamma$ in the basis, and any pair of points on this curve, the (or a) shortest path between the two points is entirely contained in $\gamma$. It follows that all curves in the basis are of length bounded by $2\, \diam(X)$.

Now let $\gamma_2$ be a curve in $\mathcal{B}$ that intersects $\gamma_1$. Again by convexity and the choice of $\gamma_1$, they intersect in exactly $1$ point, and hence lie in a one-holed torus subsurface of $X$. As a word in the fundamental group, the commutator $[\gamma_1,\gamma_2]$ corresponds to the homotopy class of the boundary of this one-holed torus. The length of its minimal geodesic representative $\delta$ is bounded (strictly) above by $2 \,\ell(\gamma_1) + 2\, \ell(\gamma_2)$. Now using the fact that $\ell(\gamma_2) \leq 2\, \diam(X)$, we have
$$
\ell(\delta) < 2\, \ls(X) + 4 \,\diam(X)
$$
and as $\delta$ is a non-trivial simple separating geodesic, the result follows.
\end{proof}

\begin{rem*}
The observation that $\sys(X) < 2 \diam(X)$ and Lemma \ref{lss<diam} imply that
\[\lss(X)<8\diam(X).\]
Now together with \cite[Part (2) of Theorem 4.10]{Mirz13}, we can deduce that
\[\frac{\E[\lss]}{\log g} \leq C_2\]
for some uniform constant $C_2>0$ independent of $g$.
\end{rem*}

\subsection{A geodesic Cheeger constant} Let $X\in \M_g$ be a hyperbolic surface. Recall that the \emph{Cheeger constant} $h(X)$ of $X$ is defined as
\[h(X)=\inf_{E\subset X_g} \frac{\ell(E)}{\min\left\{ \area(A),\area(B)\right\}}\]
where $E$ runs over all one-dimensional subsets of $X$ which divide $X$ into two disjoint components $A$ and $B$, and $\ell(E)$ is the length of $E$. 

The ``problem" with the Cheeger constant is that the set $E$ is not (necessarily) realized as a geodesic multicurve. For this reason, Mirzakhani \cite{Mirz13} introduced a \emph{geodesic Cheeger constant} $H(X)$ of $X$ defined as
\[H(X):= \inf \limits_{\gamma} \, \frac{\ell_{\gamma}(X)}{\min\left\{ \area(X_1), \area(X_2)\right\}}\]
where $\gamma$ is a multigeodesic on $X$ with $X\setminus \gamma=X_1\cup X_2$, $X_1$ and $X_2$ are connected subsurfaces of $X$ such that $\min\left\{|\chi(X_1)|, |\chi(X_2)|\right\}\geq 1$, and $\ell_{\gamma}(X)$ is the length of $\gamma$ on $X$. 

The Cheeger constant is by definition upper bounded by the geodesic Cheeger constant. Mirzakhani also provided a lower bound in the following proposition.
\begin{proposition}\cite[Proposition 4.7]{Mirz13} \label{pro-H}
Let $X\in \M_g$ be a hyperbolic surface. Then
\[\frac{H(X)}{H(X)+1}\leq h(X)\leq H(X).\]
\end{proposition}

\subsection{The Weil-Petersson metric}
Associated to a pants decomposition of $\Sigma_g$, the \emph{Fenchel-Nielsen coordinates}, given by $X \mapsto (\ell_{\alpha_i}(X),\tau_{\alpha_i}(X))_{i=1}^{3g-3}$, are global coordinates for the Teichm\"uller space $\sT_{g}$ of $\Sigma_g$. Where $\{\alpha_i\}_{i=1}^{3g-3}$ are disjoint simple closed geodesics and $\tau_{\alpha_i}$ is the twist along $\alpha_i$ (measured in length). Wolpert in \cite{Wolpert82} showed that the \wep \ sympletic structure has a natural form in Fenchel-Nielsen coordinates:
\bt[Wolpert]\label{wol-wp}
The \wep \ sympletic form $\omega_{\WP}$ on $\sT_{g}$ is given by
\[\omega_{\WP}=\sum_{i=1}^{3g-3}d\ell_{\alpha_i}\wedge d\tau_{\alpha_i}.\]
\et

We mainly work with the \emph{Weil-Petersson volume form}
$$
\dvol_{\WP}:=\tfrac{1}{(3g-3)!}\underbrace{\omega_{\WP}\wedge\cdots\wedge\omega_{\WP}}_{\text{$3g-3$ copies}}~.
$$
It is a mapping class group invariant measure on $\T_{g}$, hence is the lift of a measure on $\M_{g}$, which we also denote by $\dvol_{\WP}$. The total volume of $\M_{g}$ is finite (\emph{e.g.,} this finiteness can be obtained by using Theorem \ref{wol-wp} and the upper bound for Bers' constant \cite{Buser10}), and we denote it by $V_{g}$.

Following \cite{Mirz13}, we view a quantity $f:\M_g\to\R$ as a random variable on $\M_g$ with respect to the probability measure $\Prob$ defined by normalizing $\dvol_{\WP}$, and let $\E[f]$ denote it expectation or expected value. Namely,
$$
\Prob(\mathcal{A}):=\frac{1}{V_g}\int_{\M_g}\mathbf{1}_{\mathcal{A}}dX,\quad \E[f]:=\frac{1}{V_g}\int_{\M_g}f(X)dX,
$$
where $\mathcal{A}\subset\M_g$ is any Borel subset, $\mathbf{1}_\mathcal{A}:\M_g\to\{0,1\}$ is its characteristic function, and where $dX$ is short for $\dvol_{\WP}(X)$.

\section{Integral of $\frac{1}{h}$}\label{sec-1/h}
The main result of this section is to show that the integral of one over the Cheeger constant over a small set $\mathcal{A}_g\subset \M_g$ is also small (Proposition \ref{h-small}). In next section, we will apply Proposition \ref{h-small} to show Theorem \ref{mt-1}. 

First we recall the following well-known fact.  
\begin{lemma}\label{l-fub}
Let $(X,\mu)$ be a measure space and $f\geq 0$ on $X$ with $f\in L^1(X,\mu)$. Let $E\subset X$ be a measurable subset. Then 
$$ \int_E f(x) d\mu(x) = \int_0^\infty \mu\big(\{f(x)\geq t\}\cap E\big) dt .$$
\end{lemma}

\begin{proof}
We provide a proof for clarity. It essentially follows from Fubini's theorem:
\begin{eqnarray*}
	\int_E f(x) d\mu(x) 
	& = & \int_E \int_0^\infty \mathbf 1_{\{f(x)\geq t\}} dt d\mu(x) \\
	& = & \int_0^\infty \int_E \mathbf 1_{\{f(x)\geq t\}} d\mu(x) dt \\
	& = & \int_0^\infty \mu\big(\{f(x)\geq t\}\cap E\big) dt
\end{eqnarray*}	
which completes the proof.
\end{proof}
We now show:
\begin{proposition}\label{H-small}
Let $\sA \sbs \M_g$ be a measurable subset. Then there exists a uniform constant $c>0$ independent of $g$ and $\sA$ such that
\begin{equation*}
\frac{1}{V_g}\int_\sA \frac{1}{H(X)} dX \leq c\cdot \left(\frac{\Vol(\sA)}{V_g} + \frac{1}{g}\right).
\end{equation*}
\end{proposition}
	
\begin{proof}
First one may assume that $\Vol(\sA)>0$, otherwise there is nothing to prove. Now since $\Vol\big(\{H(X)\leq s\}\big)$ is a continuous increasing function with respect to $s$, there is a first time $t_0>0$ such that $\Vol\big(\{H(X)\leq \frac{1}{t_0}\} \big) = \Vol(\sA)$. (Note that $t_0$ may depend on $g$.) By Lemma \ref{l-fub} we have
\begin{eqnarray*}
&&\int_\sA \frac{1}{H(X)} dX = \int_0^\infty \Vol\big(\{H(X)\leq \frac{1}{t}\} \cap \sA\big) dt\\
&&= \int_0^{t_0} \Vol\big(\{H(X)\leq \frac{1}{t}\} \cap \sA\big) dt+\int_{t_0}^{\infty} \Vol\big(\{H(X)\leq \frac{1}{t}\} \cap \sA\big) dt \\
&&\leq\int_{t_0}^\infty \Vol\big(\{H(X)\leq \frac{1}{t}\}\big) dt+ t_0 \Vol(\sA).
\end{eqnarray*}

Now we use an estimate of Mirzakhani \cite[Equation (4.20)]{Mirz13} which says that there exists a uniform constant $\eps_0>0$ independent of $g$ such that for any $\eps<\eps_0$ we have
\begin{equation*}
\Vol(\{H(X)\leq \eps\}) \leq c_1\frac{\eps^2 V_g}{g}
\end{equation*}
for some uniform constant $c_1>0$ again independent of $g$. There are two cases.\\

\noindent\underline{Case 1:} Suppose that $\eps_0 \geq \frac{1}{t_0}$. 
We have
\begin{equation*}
\Vol(\sA) = \Vol\big(\{H(X)\leq \frac{1}{t_0}\} \big) \leq c_1\frac{1}{t_0^2}\frac{V_g}{g}
\end{equation*}
which implies that
\begin{eqnarray*}
	\int_\sA \frac{1}{H(X)} dX 
	&\leq& \int_{t_0}^\infty \Vol\big(\{H(X)\leq \frac{1}{t}\}\big) dt + t_0 \Vol(\sA) \\
	&\leq& \int_{t_0}^\infty c_1\frac{1}{t^2} \frac{V_g}{g} dt + c_1\frac{1}{t_0}\frac{V_g}{g} \\
	&=& 2c_1\frac{1}{t_0}\frac{V_g}{g} \\
	&\leq& 2c_1\eps_0 \frac{V_g}{g}.
\end{eqnarray*}

\noindent\underline{Case 2:} Suppose that $\eps_0 <\frac{1}{t_0}$.

In this case
\begin{eqnarray*}
	\int_\sA \frac{1}{H(X)} dX 
	&\leq& \int_{t_0}^\infty \Vol\big(\{H(X)\leq \frac{1}{t}\}\big) dt + t_0 \Vol(\sA) \\
	&\leq& \int_{t_0}^\frac{1}{\eps_0} \Vol\big(\{H(X)\leq \frac{1}{t}\}\big) dt + \int_{\frac{1}{\eps_0}}^\infty c_1\frac{1}{t^2} \frac{V_g}{g} dt + t_0 \Vol(\sA) \\
	&\leq& (\frac{1}{\eps_0} - t_0) \Vol\big(\{H(X)\leq \frac{1}{t_0}\}\big) + c_1\eps_0\frac{V_g}{g} + t_0 \Vol(\sA) \\
	&=& \frac{1}{\eps_0} \Vol(\sA) + c_1\eps_0\frac{V_g}{g}.
\end{eqnarray*}
We now choose
\[c=\max\left\{\frac{2}{\eps_0},2c_1\eps_0\right\}\]
and the conclusion follows.
\end{proof}

The following consequence of Proposition \ref{H-small} will be applied later:
\begin{proposition}\label{h-small}
Let $\sA_g \sbs \M_g$ be a measurable subset satisfying
\begin{equation*}
\limg \frac{\Vol(\sA_g)}{V_g} =0.
\end{equation*}
Then we have
\begin{equation*}
\limg \frac{1}{V_g}\int_{\sA_g} \frac{1}{h(X)} dX  =0.
\end{equation*}
\end{proposition}

\begin{proof}
Recall that \cite[Proposition 4.7]{Mirz13} tells that
	\begin{equation*}
	\frac{H(X)}{H(X)+1} \leq h(X) \leq H(X)
	\end{equation*}
which together with Proposition \ref{H-small} imply that
\begin{eqnarray*}
\limg \frac{1}{V_g}\int_{\sA_g} \frac{1}{h(X)} dX&\leq& \limg \frac{1}{V_g}\int_{\sA_g}\left(1+ \frac{1}{H(X)}\right) dX \\
&=&0.
\end{eqnarray*}

The proof is complete.
\end{proof}

The following result is contained in the proof of \cite[Theorem 4.10]{Mirz13}. For completeness we outline a proof here.
\begin{proposition}\label{l1-h}
Let $\ell_1(X)=\min\{\ell_{sys}(X),1\}$. Then there exists a uniform constant $c>0$ independent of $g$ such that
\[\frac{\int_{\sM_g} \frac{|\log (\ell_1(X))|}{h(X)}dX}{V_g}\leq c.\]
\end{proposition}

\bp
First we recall two results from \cite{Mirz13}. By setting $\beta=\frac{4}{3}$ and applying \cite[Theorem 4.8]{Mirz13}, we have that there exist two uniform constants $C_1,C_2>0$ independent of $g$ such that
\begin{equation}\label{1-h-v}
C_1 V_g \leq \int_{\sM_g} \frac{1}{h(X)^{\frac{4}{3}}}dX \leq C_2 V_g.
\end{equation}
By \cite[Corollary 4.3]{Mirz13} we know that there exist two uniform constants $C_3,C_4>0$ independent of $g$ such that
\begin{equation}\label{1-sys-v}
C_3 V_g\leq \int_{\sM_g} \frac{1}{\ell_{sys}(X)}dX\leq C_4 V_g \nonumber
\end{equation}
which in particular implies that for any $a>0$
\begin{equation}\label{1-sys-v-2}
\int_{\sM_g} |\log (\ell_1(X))|^a dX=O(V_g)
\end{equation}
where the implied constant is independent of $g$.

The rough equalities \eqref{1-h-v} and \eqref{1-sys-v-2} together with H\"older's inequality imply that 
\begin{eqnarray*}
\int_{\sM_g} \frac{|\log (\ell_1(X))|}{h(X)}dX &\leq& \left( \int_{\sM_g} \frac{1}{h(X)^{\frac{4}{3}}}dX \right)^{\frac{3}{4}}\cdot \left( \int_{\sM_g} |\log (\ell_1(X))|^4 dX\right)^{\frac{1}{4}}\\
&=&O(V_g)
\end{eqnarray*}
as required.
\ep

\section{Proof of Theorem \ref{mt-1}}\label{Sec-mt-1}
In this section we complete the proof of our main result, Theorem \ref{mt-1}.

We recall two prior results essential for our purposes. The first one is:
\begin{thm*}\cite[Theorem 4.4]{Mirz13}
Let $0<a<2$. Then
\[\Prob\left(X\in \M_g \mid \ \lss(X)<a \log g \right)=O\left( \frac{(\log g)^3 g^{\frac{a}{2}}}{g} \right).\]
\end{thm*}
\noindent This result in particular implies that for any $\epsilon>0$,
\be \label{lower-ss}
\lim \limits_{g\to \infty} \Prob \left(X\in \sM_g \mid \ \lss(X)>(2-\epsilon)\log g \right)=1.
\ene
The second one is:
\begin{thm*} \cite[Theorem 1]{NWX20}
Let $\omega(g)$ be a function satisfying
\be \label{eq-omega}
\lim \limits_{g\to \infty}\omega(g)= +\infty \ \textit{and} \ \lim \limits_{g\to \infty}\frac{\omega(g)}{\log\log g} = 0.\nonumber
\ene
Consider the following conditions on surfaces $X\in\M_g$:
\begin{itemize}
\item[(a)] \label{item_main1} $|\ell_{\sys}^{\rm sep}(X)-(2\log g - 4\log \log g)| \leq \omega(g)$;
\item[(b)] \label{item_main2} $\ell_{\sys}^{\rm sep}(X)$ is achieved by a simple closed geodesic separating $X$ into $S_{1,1}\cup S_{g-1,1}$.
\end{itemize}

Then we have
$$
\lim \limits_{g\to \infty} \Prob\left(X\in \M_g \mid \ \textit{$X$ satisfies $(a)$ and $(b)$} \right)=1.
$$ 
\end{thm*}

\noindent In particular, this implies that for any $\epsilon>0$ 
\be \label{upp-ss}
\lim \limits_{g\to \infty} \Prob\left(X\in \M_g \mid \ (2-\epsilon)\log g< \lss(X)<2\log g \right)=1.
\ene 

We now define a special subset of moduli space $\mathcal{B}_g$:
\[\mathcal{B}_g:=\{X\in \sM_g \mid \ \lss(X)<2\log g\}.\]

Note that \eqref{upp-ss} implies 
\begin{equation}\label{nwx-sep}
\lim \limits_{g\to \infty} \Prob \left(X\in \mathcal{B}_g \right)=1.
\end{equation}
Now we are ready to show our main result. 
\begin{proof}[Proof of Theorem \ref{mt-1}]
First we recall the following result of Brooks in \cite{Brooks90} (or see \cite[Equation (4.21)]{Mirz13}), which for any $X\in \sM_g$ bounds $\diam(X)$ from above in terms of the systole and the Cheeger constant:
\[\diam(X)\leq 2 \left(\ls(X)/2+\frac{1}{h(X)}\cdot \log \left(\frac{2\pi(g-1)}{\area (B_{\H}(\ls(X)/2))} \right) \right).\]
(The quantity $\area (B_{\H}(\ls(X)/2))$ is the hyperbolic area of a geodesic ball of radius $\ls(X)/2$ in the hyperbolic plane $\H$.) 

Recall that $\ell_1(X)=\min\{\ell_{sys}(X),1\}$ and that $\ls(X)=O(\log g)$. The above result of Brooks together with Proposition \ref{lss<diam} imply that
\begin{equation}\label{brooks}
\lss(X)=O\left(\log g+\frac{\log g}{h(X)}+\frac{|\log (\ell_1(X))|}{h(X)}\right)
\end{equation}
where the implied constant is independent of $g$.

Since $\eps>0$ is arbitrary, it follows from \eqref{lower-ss} that 
\be
\liminf \limits_{g\to \infty} \frac{\ewp[\lss]}{\log g}\geq 2.
\ene

Now we provide the other inequality. Let $\mathcal{A}_g$ be the complement of $\mathcal{B}_g$ in $\sM_g$:
\be
\sA_g:=\{X\in \sM_g \mid \ X \notin \mathcal{B}_g\}.\nonumber
\ene
So by \eqref{nwx-sep} we know that
\begin{equation}\label{0mea}
\limg \frac{\Vol(\sA_g)}{V_g} =0.
\end{equation}
We split the integral as
\begin{eqnarray*}
\frac{\ewp[\lss]}{\log g}=\frac{\int_{\mathcal{B}_g} \lss(X) dX}{\log g \cdot V_g}+\frac{\int_{\mathcal{A}_g} \lss(X) dX}{\log g \cdot V_g}.
\end{eqnarray*}
By \eqref{brooks}, there exists a uniform constant $c>0$ independent of $g$ such that
\begin{eqnarray*}
\frac{\ewp[\lss]}{\log g}<2+c\cdot \left( \frac{\Vol(\sA_g)}{V_g} + \frac{\int_{\sA_g} \frac{1}{h(X)} dX }{V_g}+\frac{1}{\log g} \frac{\int_{\sM_g} \frac{|\log (\ell_1(X))|}{h(X)}dX}{V_g}\right).
\end{eqnarray*}
Let $g\to \infty$, then it follows by \eqref{0mea}, Proposition \ref{h-small} and Proposition \ref{l1-h} that
\be
\limsup \limits_{g\to \infty} \frac{\ewp[\lss]}{\log g}\leq 2.
\ene
which completes the proof.\end{proof}

\bibliographystyle{plain}
\bibliography{wp}
	
\end{document}